\documentclass[]{amsart}
\usepackage{amsmath}
\usepackage[english]{babel}
\usepackage[utf8]{inputenc}
\usepackage{amsfonts}
\usepackage{amsthm}
\usepackage{color}
\numberwithin{equation}{section}

\newtheorem{thm}{Theorem}[section]

\newtheorem{prop}[thm]{Proposition}
\newtheorem{cor}[thm]{Corollary}
\newtheorem{lem}[thm]{Lemma}

\theoremstyle{remark}
\newtheorem{rmk}[thm]{Remark}
\theoremstyle{definition}

\DeclareMathOperator{\Cov}{Cov}
\DeclareMathOperator{\E}{\mathbb{E}}
\DeclareMathOperator{\N}{\mathbb{N}}
\DeclareMathOperator{\R}{\mathbb{R}}

\DeclareMathOperator{\cG}{\mathcal{G}}

\DeclareMathOperator{\cA}{\mathcal{A}}
\DeclareMathOperator{\cL}{\mathcal{L}}

\DeclareMathOperator{\D}{\mathbb{D}}

\DeclareMathOperator{\bP}{\mathbb{P}}

\DeclareMathOperator{\mm}{\mathbf{m}}

\newcommand{\pd}[2]{\frac{\partial #1}{\partial #2}}

\newcommand{\der}[2]{\frac{d #1}{d #2}}

\newcommand{\derfrac}[2]{\frac{d^{\nu} #1}{d #2^{\nu}}}

\newcommand{\Norm}[2]{\left\Vert #1 \right\Vert_{#2}}

\title{Fractional Immigration-Death Processes}
\author{Giacomo Ascione$^\ast$}
\address{$^\ast$ Dipartimento di Matematica e Applicazioni ``Renato Caccioppoli'', Università degli Studi di Napoli Federico II, 80126 Napoli, Italy}
\author{Nikolai Leonenko$^\dagger$}
\thanks{N. Leonenko was supported in particular by Australian Research Council's Discovery Projects funding scheme (project DP160101366), and  by project MTM2015-71839-P of MINECO, Spain (co-funded with FEDER funds).}
\address{$^\dagger$ School of Mathematics, Cardiff University, Cardiff CF24 4AG, UK}
\author{Enrica Pirozzi$^\ast$}
\email{giacomo.ascione@unina.it \\
	leonenkon@cardiff.ac.uk \\
	enrica.pirozzi@unina.it}

\begin{document}
\maketitle
\begin{abstract}
	In this paper we study explicit strong solutions for two difference-differential fractional equations, defined via the generator of an immigration-death process, by using spectral methods. Moreover, we give a stochastic representation of the solutions of such difference-differential equations by means of a stable time-changed immigration-death process and we use this stochastic representation to show boundedness and then uniqueness of these strong solutions. Finally, we study the limit distribution of the time-changed process.
\end{abstract}
\keywords{Stable subordinator, Caputo Fractional derivative, Time-changed process, Birth-death process}

\section{Introduction}
Birth-death processes constitute an important class of continuous time Markov chain (CTMC). They are widely used, for instance, in population and evolutionary dynamics (see \cite{Novozhilov2006,Nowak2006}), queueing theory (see \cite{Sharma2000}) and in epidemiology (see \cite{Allen2015}). A complete classification and characterization of birth-death processes is due to Karlin and McGregor, whose papers \cite{KarlinMcGregor1957a,KarlinMcGregor1957b} are the starting point of the study of family of classical orthogonal polynomials linked to such processes.\\
Classical orthogonal polynomials are widely used to study the solutions of Kolmogorov equations as in the case in which the state space of the process is continuous, as well as in the discrete one. In the continuous case, the families of classical orthogonal polynomials are used to give a spectral decomposition of Kolmogorov equations induced by the generators of Pearson diffusions \cite{Forman2008}. In the discrete case, the discrete analogue of Pearson diffusions is given by a certain class of solvable birth-death processes \cite{Kuznetsov2004}. Moreover one can associate to any family of classical orthogonal polynomials of discrete variable another particular family, called the dual family \cite{NikiforovUvarovSuslov1991}. In some cases, a family of classical orthogonal polynomials of discrete variable could be in duality with itself: in this case it is called \textit{self-dual} family \cite{Schoutens2012}. Among self-dual families, the simplest one is the family of Charlier polynomials, whose self-duality is induced by the following formula
\begin{equation*}
C_n(x,\alpha)=C_x(n,\alpha) , \ n,x \in \mathbb{N}_0,
\end{equation*}
called duality formula for Charlier polynomials (see Section \ref{sec3} for the definition of Charlier polynomials).\\
Charler polynomials are really useful in the study of immigration-death processes (or $M/M/\infty$ queues) \cite{Schoutens2012} and in their general version on $1$-dimensional lattice, called Charlier processes \cite{AlbaneseKuznetsov2005,Kuznetsov2004}. Indeed, one can give a spectral decomposition of the strong solutions of Kolmogorov equations induced by the generator of the immigration-death processes in terms of such polynomials.\\
For Pearson diffusions, the classical orthogonal polynomials are powerful tools to study strong solutions of fractional Kolmogorov equations and characterize a stochastic representation of such solutions via time-changed (through the inverse of a L\`{e}vy subordinator) Markov processes \cite{Gajda,Leonenko2013a,Leonenko2013b,Leonenko2017}. In the discrete case, fractional (time-changed) processes have been widely considered via different approaches. First of all, a fractional version of the Poisson process has been introduced using Mittag-Leffler distributed inter-jump times instead of exponential ones \cite{Aletti2018,Laskin2003, Leonenko2019, MainardiGorenfloScalas2007,MainardiGorenfloVivoli2005} (this approach has been also applied to general counting processes \cite{DiCrescenzo2015}). Such process can be also obtained using a fractional differential-difference equations approach \cite{BeghinOrsingher2009,BeghinOrsingher2010} and by means of a time-change \cite{MeerschaertNaneVellaisamy2011}.\\
With the same approach, some classes of fractional birth-death processes have been introduced and studied \cite{Orsingher2010a,Orsingher2010b,Orsingher2011}: in these papers, properties of these processes are deduced from a fractional version of their Kolmogorov forward equation.\\
Here, following the approach of \cite{Leonenko2013a}, we show the existence of strong solutions for the time-fractional counterpart of the Kolmogorov backward and forward equations of immigration-death processes with the aid of Charlier polynomials and link them to a time-changed immigration-death process.\\
In particular:
\begin{itemize}
	\item in Section \ref{sec2} we give some basics on birth-death processes;
	\item in Section \ref{sec3} we give some notions on the classical immigration-death process, defining its generator and its forward operator;
	\item in Section \ref{sec4} we show the existence of strong solutions of the time-fractional Kolmogorov backward and forward equations under suitable assumption on the initial data;
	\item in Section \ref{sec5} we introduce a fractional immigration-death process and show how the strong solutions of the time-fractional Kolmogorov backward and forward equations can be interpreted by using such process;
	\item in Section \ref{sec6}, we show the uniqueness of such strong solutions by using the aforementioned stochastic representation and a uniqueness criterion for uniformly bounded solutions \cite{AscioneLeonenkoPirozzi2018}, under suitable assumptions on the initial data;
	\item finally, in Section \ref{sec7} we give the limit distribution of the constructed fractional immigration-death process and we discuss its autocovariance function.
\end{itemize}
\section{Birth-death processes}\label{sec2}
Let us give some information about general birth-death processes, following the lines of \cite{KarlinMcGregor1957a,KarlinMcGregor1957b}. We say that a time-homogeneous continuous time Markov chain $N(t)$ defined on $\mathbb{N}_0=\{0,1,2,\dots\}$ is a birth and death process if and only if, denoting with 
\begin{equation*}
p(t,x;y)=\bP(N(t+s)=x|N(s)=y), \ x,y=0,1,2,\dots; \ t,s\ge 0,
\end{equation*}
the transition probability functions and $P(t)=(p(t,x;y))_{x,y\ge 0}$ the transition probability matrix, it is solution of the following two differential equations
\begin{align}
P'(t)=\cA P(t),\label{eq:backandfor} \qquad
P'(t)=P(t)\cA,
\end{align}
with initial condition $P(0)=I$ and the infinite matrix $\cA=(A(x,y))_{x,y\ge 0}$ is such that:
\begin{align*}
&A(x,x+1)=B(x) \quad x \ge 0, \qquad &A(x,x)=-(B(x)+D(x)) \quad x \ge 0,\\
&A(x,x-1)=D(x) \quad x \ge 1, \qquad
&A(x,y)=0 \quad |x-y|>1,
\end{align*}
where $B(x)>0$ for any $x \ge 0$, $D(x)>0$ for any $x \ge 1$ and $D(0)\ge 0$. Equations \eqref{eq:backandfor} are called respectively \textit{backward and forward Kolmogorov equation}. In order to obtain $P(t)$ we need to impose other two properties:
\begin{align*}
P_{i,j}(t)\ge 0, \qquad \sum_{j=0}^{+\infty}P_{i,j}(t)\le 1.
\end{align*}
In particular it is possible to show that $N(t)$ is a birth-death process if and only if its generator is given by:
\begin{align*}
\begin{split}
\cG f(x)&=(B(x)-D(x))\nabla^+f(x)+D(x)\Delta f(x)\\&=(B(x)-D(x))\nabla^-f(x)+B(x)\Delta f(x),
\end{split}
\end{align*}
for $x=0,1,2,\dots$ and $f(-1)=0$, where the difference-type operators $\nabla^\pm$ and $\Delta$ are defined as
\begin{align*}
\nabla^+ f(x)&=f(x+1)-f(x) \ \forall x \in \mathbb{N}_0\\
\nabla^- f(x)&=f(x)-f(x-1) \ \forall x \in \mathbb{N}_0\\
\Delta f(x)&=f(x+1)-2f(x)+f(x-1) \ \forall x \in \mathbb{N}_0.
\end{align*}
The following discrete versions of the Leibnitz rule will be useful
\begin{align}\label{eq:forLeibrule}
\nabla^+(fg)(x)&=f(x+1)\nabla^+ g(x)+g(x)\nabla^+f(x)\\
\label{eq:backLeibrule}
\nabla^-(fg)(x)&=f(x)\nabla^- g(x)+g(x-1)\nabla^-f(x)\\
\label{eq:LapLeibrule} \Delta(fg)(x)&=f(x+1)\nabla^+ g(x)-f(x-1)\nabla^-g(x)+g(x)\Delta f(x).
\end{align}
The backward Kolmogorov equation becomes, for fixed $x \in \mathbb{N}_0$
\begin{equation*}
\begin{cases}
p'(t,x;y)=\cG p(t,x;y)\\
p(0,x;y)=\delta_{x,y},
\end{cases}
\end{equation*}
where $\cG$ works on $y$ and 
\begin{equation*}
\delta_{x,y}=\begin{cases} 1 & x=y \\
0 & \mbox{otherwise}.\end{cases}
\end{equation*}
is Kronecker symbol.\\
Moreover we can find a \textit{forward} operator
\begin{align*}
\begin{split}
\cL f(x)&=-\nabla^-((B(\cdot)-D(\cdot))f)(x)+\Delta(D(\cdot)f)(x)\\&=-\nabla^+((B(\cdot)-D(\cdot))f)(x)+\Delta(B(\cdot)f)(x),
\end{split}
\end{align*}
so that for fixed $y \in \mathbb{N}_0$ the forward Kolmogorov equation becomes
\begin{equation*}
\begin{cases}
p'(t,x;y)=\cL p(t,x;y)\\
p(0,x;y)=\delta_{x,y},
\end{cases}
\end{equation*}
where $\cL$ works on $x$.\\
We will focus on the case in which the generator is in the form:
\begin{equation*}
\cG=p_1(x)\nabla^++p_2(x)\Delta,
\end{equation*}
where $p_1(x)$ and $p_2(x)$ are polynomials such that $\deg p_1(x)\le 1$ and $\deg p_2(x) \le 2$. Then we can find the \textit{classical orthogonal polynomials of discrete variable} as solution of the equation
\begin{equation*}
\cG f(x)=-\lambda f(x),
\end{equation*}
for some $\lambda$, which is an hypergeometric type difference equation. The values that these polynomials assume on a lattice $\{D_1,D_1+1,\dots,D_2\}$ for some $D_1,D_2$ fully characterize the transition probability and the solutions of the backward and forward Kolmogorov equations. Moreover, these polynomials respect an orthogonality relation in $\ell^2(\mm)$ for some measure $\mm$ called the spectral measure, which is an atomic measure on the lattice. In this case, the spectral measure coincides with the invariant measure of the process $N(t)$ and its mass function $m(x)=\mm(\{x\})$ is solution of a discrete analogue of the Pearson equation
\begin{equation*}
\nabla^+(p_2(x)\mm(x))=p_1(x)\mm(x).
\end{equation*}
Following the lines of \cite{Kuznetsov2004}, for $p_1(x)=a-bx$, we can recognize the following three class of solvable birth-death processes:
\begin{itemize}
	\item For $p_2(x)=bx$ we have the Immigration-Death process;
	\item For $p_2(x)=\frac{1}{2}\sigma^2x$ where $\frac{1}{2}\sigma^2 \not = b$ we have a negative binomial process;
	\item For $p_2(x)=\frac{1}{2}\sigma^2x(A-x)$ we have a hypergeometric process.
\end{itemize}
However, we will focus only on the first case for the choice of the polynomials $p_1$ and $p_2$.
\section{Immigration-death processes}\label{sec3}
Fix $a,b>0$ the operator
\begin{equation*}
\cG=(a-bx)\nabla^-+a\Delta;
\end{equation*}
which is a discrete version of the Ornstein-Uhlenbeck generator on $\mathbb{N}_0$.\\
A continuous time Markov chain $N(t)$ defined on $\mathbb{N}_0$ that admits $\cG$ as generator will be called \textit{immigration-death process} (or also $M/M/\infty$ queue: see, for instance, \cite{Schoutens2012}).
This process can be generalized to a particular birth-death process with values on a $1$-dimensional lattice called \textit{Charlier process} (see \cite{AlbaneseKuznetsov2005}), but we will focus on the $\mathbb{N}_0$-valued one.
For such process, the backward Kolmogorov equations are in the form
\begin{equation*}
\der{u}{t}(t,x)=\cG u(t,x).
\end{equation*}
Moreover, from $\cG$ we can recognize the birth and death parameters as
\begin{align*}
B(x)=a, && D(x)=bx,
\end{align*}
and thus the forward operator as
\begin{equation*}
\cL f(x)=-\nabla^+((a-bz)f(z))(x)+a\Delta f(x),
\end{equation*}
where with $\nabla^+((a-bz)f(z))(x)$ we intend the operator $\nabla^+$ applied to the function $z \mapsto (a-bz)f(z)$ and then evaluated in $x$.\\
The operators $\cG$ and $\cL$ can be represented as infinite matrices. In particular we have $\cG=(G(x,y))_{x,y\ge 0}$ where, for $x>0$
\begin{equation*}
\begin{gathered}
G(x,x-1)=bx \qquad G(x,x)=-(a+bx) \qquad G(x,x+1)=a\\
G(0,0)=-a \qquad G(0,1)=a
\end{gathered}
\end{equation*}
and $\cL=(L(x,y))_{x,y \ge 0}$ where, for $x>0$
\begin{equation*}
\begin{gathered}
L(x,x-1)=a \qquad L(x,x)=-(a+bx) \qquad L(x,x+1)=b(x+1)\\
L(0,0)=-a \qquad L(0,1)=b.
\end{gathered}
\end{equation*}
The stationary measure of the process $N(t)$ is the Poisson distribution of parameter $\alpha$, given by:
\begin{equation*}
\mm(\{x\})=e^{-\alpha}\frac{\alpha^x}{x!}, \ x=0,1,2,\dots.
\end{equation*}
Now let us introduce the main Banach sequence spaces we will use through this paper:
\begin{itemize}
	\item Let us denote with $\ell^\infty$ the Banach space of bounded functions $f:\mathbb{N}_0 \to \R$ equipped with the norm
	\begin{equation*}
	\Norm{f}{\ell^\infty}=\max_{x \in \mathbb{N}_0}|f(x)|;
	\end{equation*}
	\item Let us denote with $c_0$ the subspace of $\ell^\infty$ of bounded functions $f:\mathbb{N}_0 \to \R$ such that $\lim_{x}f(x)=0$;
	\item Let us denote with $\ell^1$ the Banach space of the functions $f:\mathbb{N}_0 \to \R$ such that
	\begin{equation*}
	\Norm{f}{\ell^1}=\sum_{x=0}^{+\infty}|f(x)|<+\infty.
	\end{equation*}
	\item Let us denote with $\ell^2$ the Hilbert space of functions $f:\mathbb{N}_0 \to \R$ such that
	\begin{equation*}
	\Norm{f}{\ell^2}^2:=\sum_{x=0}^{+\infty}f^2(x)<+\infty
	\end{equation*}
	equipped with the scalar product
	\begin{equation*}
	\langle f,g\rangle_{\ell^2}=\sum_{x=0}^{+\infty}f(x)g(x)
	\end{equation*}
	\item Let us denote with $\ell^2(\mm)$ the Hilbert space of functions $f:\mathbb{N}_0 \to \R$ such that
	\begin{equation*}
	\Norm{f}{\ell^2(\mm)}^2:=\sum_{x=0}^{+\infty}\mm(\{x\})f^2(x)<+\infty
	\end{equation*}
	equipped with the scalar product
	\begin{equation*}
	\langle f,g\rangle_{\ell^2(\mm)}=\sum_{x=0}^{+\infty}\mm(\{x\})f(x)g(x).
	\end{equation*}
\end{itemize}
\begin{rmk}\label{rmk:emb}
	Let us observe that $\ell^2$ is continuously included in $\ell^2(\mm)$. Consider a function $f \in \ell^2$. Then
	\begin{equation*}
	\sum_{x=0}^{+\infty}m(x)f^2(x)=e^{-\alpha}\sum_{x=0}^{+\infty}\frac{\alpha^x}{x!}f^2(x).
	\end{equation*}
	Now, let us observe that the sequence $x\mapsto\frac{\alpha^x}{x!}$ converges to $0$ as $x \to +\infty$, hence there exists a constant $C(\alpha)$ such that $\frac{\alpha^x}{x!}\le C(\alpha)$. Thus
	\begin{equation*}
	e^{-\alpha}\sum_{x=0}^{+\infty}\frac{\alpha^x}{x!}f^2(x)\le e^{-\alpha}C(\alpha)\Norm{f}{\ell^2}^2.
	\end{equation*}
	Moreover, since $\ell^1$ is continuously included in $\ell^2$ (see \cite{Villani1985}), we have that $\ell^1$ is also continuously included in $\ell^2(\mm)$.
\end{rmk}
From the matrix representation of the generator $\cG$ and the forward operator $\cL$ one can prove the following Lemma.
\begin{lem}\label{lem:contlemma}
	The operators $\cG:\ell^2(\mm)\mapsto \ell^2(\mm)$ and $\cL:\ell^2(\mm)\mapsto \ell^2(\mm)$ are continuous.
\end{lem}
\begin{proof}
	The proof is a straightforward consequence of Schur's test (see \cite{HalmosSunder2012})
\end{proof}
Moreover, another interesting property that follows from the matrix representation of $\cG$ is given by the following Lemma.
\begin{lem}\label{lem:Fellemma}
	The process $N(t)$ is a Feller process.
\end{lem}
\begin{proof}
	The proof is a straightforward consequence of \cite[Corollary $3.2$]{Eth2005}.
\end{proof}
Let us also observe that the spectrum of $\cG$ is given by the sequence $\lambda_n=-bn$ (\cite{Kuznetsov2004}), while the eigenfunctions are defined as $x \mapsto C_n(x,\alpha)$ where $\alpha=\frac{a}{b}$ and $C_n$ are the Charlier polynomials (see \cite{NikiforovUvarovSuslov1991,Schoutens2012}), which are defined by the generating function
\begin{equation*}
\sum_{n=0}^{+\infty}C_n(x,\alpha)\frac{t^n}{n!}=e^{-t}\left(1+\frac{t}{\alpha}\right)^x, \quad \ t \in \R
\end{equation*}
or via the three terms recurrence relations:
\begin{equation*}
-xC_n(x,\alpha)=\alpha C_{n+1}(x,\alpha)-(n+\alpha)C_n(x,\alpha)+nC_{n-1}(x,\alpha), \ n \ge 0,
\end{equation*}
where $C_{0}(x,\alpha)\equiv 1$ and $C_{-1}(x,\alpha)\equiv 0$, or
\begin{equation*}
C_{n+1}(x,\alpha)=\frac{1}{\alpha}\left[xC_n(x-1,\alpha)-C_n(x,\alpha)\right].
\end{equation*}
The first few Charlier polynomials are
\begin{align*}
C_0(x,\alpha)=1, \qquad C_1(x,\alpha)=\frac{x}{\alpha}-1, \qquad
C_2(x,\alpha)=\frac{x(x-1)}{\alpha^2}-2\frac{x}{\alpha}+1, \ \dots
\end{align*}
The orthogonality relation between the polynomials $C_n$ is given by
\begin{equation*}
\sum_{x=0}^{+\infty}C_n(x,\alpha)C_m(x,\alpha)\mm(\{x\})=\frac{n!}{\alpha^n}\delta_{n,m},
\end{equation*}
where $\delta_{n,m}$ is the Kronecker delta symbol. Thus, posing $d_n^2=\frac{n!}{\alpha^n}$, we have that
\begin{equation*}
\Norm{C_n(\cdot,\alpha)}{\ell^2(\mm)}=d_n.
\end{equation*}
Let us then define an orthonormal system of polynomials given by
\begin{equation}\label{eq:Q}
Q_n(x)=\frac{C_n(x,\alpha)}{d_n}.
\end{equation}
Let us also recall that we can exploit the decomposition of a function $g \in \ell^2(\mm)$ by means of the orthonormal basis $\{Q_n\}_{n \in \mathbb{N}_0}$. Indeed for any $g \in \ell^2(\mm)$, given the decomposition $g(x)=\sum_{n=0}^{+\infty}g_nQ_n(x)$ where $g_n=\langle g, Q_n \rangle_{\ell^2(\mm)}$, the sequence $\{g_n\}_{n \in \mathbb{N}_0}\in \ell^2$.\\
By using such orthonormal system of polynomials, it is well known (see \cite{KarlinMcGregor1957a,KarlinMcGregor1957b} but also \cite{Kuznetsov2004} for a review) that the transition probability function of the immigration-death process is given by
\begin{equation*}
p(t,x_1;x_0)=m(x_1)\sum_{n=0}^{\infty}e^{-bnt}Q_n(x_0)Q_n(x_1),
\end{equation*}
where $m(x)=\mm(\{x\})$ and is the fundamental solution of the backward Kolmogorov equation, that is to say that the Cauchy problems
\begin{equation*}
\begin{cases}
\der{u}{t}(t,x)=\cG u(t,x) \\
u(0,x)=g(x),
\end{cases}
\end{equation*}
and
\begin{equation*}
\begin{cases}
\der{v}{t}(t,x)=\cL v(t,x) \\
v(0,x)=f(x),
\end{cases}
\end{equation*}
with $g,f/m \in \ell^2(\mm)$ admit strong solutions $v$ given by
\begin{equation}\label{eq:stsolnofrac}
u(t,x)=\sum_{y=0}^{+\infty}p(t,y;x)g(y)=\sum_{n=0}^{+\infty}g_ne^{-bnt}Q_n(x),
\end{equation}
and
\begin{equation*}
v(t,x)=\sum_{y=0}^{+\infty}p(t,x;y)f(y)=m(x)\sum_{n=0}^{+\infty}f_ne^{-bnt}Q_n(x),
\end{equation*}
where $g(x)=\sum_{n=0}^{+\infty}g_nQ_n(x)$ and $f(x)/m(x)=\sum_{n=0}^{+\infty}f_nQ_n(x)$ and the convergence is uniform. In particular from \eqref{eq:stsolnofrac} one easily obtains that
\begin{equation}\label{eq:condmean}
\E[N(t)|N(0)=x]=xe^{-bt}+\alpha(1-e^{-bt}).
\end{equation}
\section{Strong solutions in the fractional case}\label{sec4}
Let us introduce the fractional derivative operator (see \cite{LiQianChen2011}). Fix $\nu \in (0,1)$ and consider the Caputo fractional derivative given by
\begin{equation}\label{eq:caputoder}
\frac{\partial^\nu u}{\partial t^\nu}(t,x)=\frac{1}{\Gamma(1-\nu)}\left[\pd{}{t}\int_0^t(t-\tau)^{-\nu}u(\tau,x)d\tau-\frac{u(0,x)}{t^\alpha}\right],
\end{equation}
that, if $u$ is differentiable in $t$, can be written also as
\begin{equation*}
\frac{\partial^\nu u}{\partial t^\nu}(t,x)=\frac{1}{\Gamma(1-\nu)}\int_0^t(t-\tau)^{-\nu}\pd{u}{t}(\tau,x)d\tau,
\end{equation*}
and pose for $\nu=1$ $\frac{\partial^\nu u}{\partial t^\nu}=\frac{\partial u}{\partial t}$. Note that the classes of functions for which the Caputo fractional derivative is well defined are discussed in \cite[Section $2.2$ and $2.3$]{MeerschaertSikorskii2011} (in particular one can use the class of absolutely continuous functions).\\
Denote with 
\begin{equation*}
\widetilde{u}(s,x)=\int_0^{+\infty}e^{-st}u(t,x)dt, \ s>0
\end{equation*}
the one-sided Laplace transform of $u$ with respect to $t$. Thus we have that the Laplace transform of $\frac{\partial^\nu u}{\partial t^\nu}$ is given by 
$$s^\nu \widetilde{u}(s,x)-s^{\nu-1}\widetilde{u}(0,x).$$\\
We want to find strong solutions for fractional Cauchy problems in the form:
\begin{equation}\label{eq:prob}
\begin{cases}
\frac{\partial^\nu u}{\partial t^\nu}(t,x)=\cG u(t,x);\\
u(0,x)=g(x),
\end{cases}
\end{equation}
for $g \in \ell^2(\mm)$ with the decomposition $g(x)=\sum_{n=0}^{+\infty}g_nQ_n(x)$. The main idea is to find a solution via separation of variables. Indeed, we can suppose that $u(t,x)=T(t)\varphi(x)$ and then observing that, if $u$ is solution of the first equation of \ref{eq:prob}, then
\begin{equation*}
\varphi(x)\derfrac{T}{t}(t)=T(t)\cG \varphi(x),
\end{equation*}
that leads, if $\varphi$ and $T$ do not vanish, to the two coupled equations:
\begin{equation*}
\begin{cases}
\cG \varphi(x)=-\lambda \varphi(x),\\
\derfrac{T}{t}(t)=-\lambda T(t),
\end{cases}
\end{equation*}
which are two eigenvalue problems. In particular we have observed that the first one admits a non zero solution if and only if $\lambda=-bn$ for some $n \in \mathbb{N}_0$ and in that case we can consider $\varphi(x)=Q_n(x)$. Moreover, the second problem admits a solution in the form
\begin{equation*}
T(t)=E_\nu(-\lambda t^\nu),
\end{equation*}
where $E_\nu$ is the Mittag-Leffler function defined as
\begin{equation}\label{eq:ML}
E_\nu(z)=\sum_{j=0}^{+\infty}\frac{z^j}{\Gamma(1+\nu j)}, \ z \in \mathbb{C}
\end{equation} 
(see, for instance, \cite{KilbasSrivastavaTrujillo2006}). Thus the idea is to find a solution in the form
\begin{equation*}
u(t,x)=\sum_{n=0}^{+\infty}u_nE_\nu(-bnt^\nu)Q_n(x).
\end{equation*}
Moreover, the initial condition suggests that
\begin{equation*}
\sum_{n=0}^{+\infty}u_nQ_n(x)=\sum_{n=0}^{+\infty}g_nQ_n(x),
\end{equation*}
so we have $u_n=g_n$ and then we expect the solution to be
\begin{equation}\label{eq:specdecheur}
u(t,x)=\sum_{n=0}^{+\infty}g_nE_\nu(-bnt^\nu)Q_n(x).
\end{equation}
These heuristic arguments have shown us how should the solution look like, hence we have to prove that such function $u$ is the solution we are searching for.\\
With the following Lemma, we will first exhibit the fundamental solution of the fractional Cauchy problem in Eq. \eqref{eq:prob}. 
\begin{lem}\label{lem:fundsol}
	Consider the series
	\begin{equation}\label{eq:fundsol}
	p_\nu(x,t;y)=m(x)\sum_{n=0}^{+\infty}E_\nu(-bnt^\nu)Q_n(x)Q_n(y),
	\end{equation}
	where $Q_n$ and $E_\nu$ are the functions defined in Equations \eqref{eq:Q} and \eqref{eq:ML} and $m(x)=\mm(\{x\})$. Then such series converges for fixed $t>0$ and $x,y \in \mathbb{N}_0$.
\end{lem}
\begin{proof}
	To show the convergence of $p_\nu(x,t;y)$, we need the following \textit{self-duality} property of the Charlier polynomials (see \cite[Equation $2.7.10a$]{NikiforovUvarovSuslov1991}):
	\begin{equation}\label{eq:selfdual}
	C_n(x,\alpha)=C_x(n,\alpha), \ \forall n,x \in \mathbb{N}_0.
	\end{equation}
	From this relation we have
	\begin{align*}
	p_\nu(x,t;y)&=m(x)\sum_{n=0}^{+\infty}E_\nu(-bnt^\nu)Q_n(x)Q_n(y)\\&=m(x)\sum_{n=0}^{+\infty}\frac{1}{d_n^2}E_\nu(-bnt^\nu)C_n(x,\alpha)C_n(y,\alpha)\\&=m(x)\sum_{n=0}^{+\infty}\frac{1}{d_n^2}E_\nu(-bnt^\nu)C_x(n,\alpha)C_y(n,\alpha),
	\end{align*}
	hence we need to show the convergence of the series
	\begin{equation*}
	\sum_{n=0}^{+\infty}\frac{1}{d_n^2}E_\nu(-bnt^\nu)C_x(n,\alpha)C_y(n,\alpha).
	\end{equation*}
	Now, let us observe that equation \eqref{eq:selfdual} made us fix the degrees of the polynomials involved in the series. Thus, let us denote with $z_x$ and $z_y$ the last real zeroes of $C_x(\cdot,\alpha)$ and $C_y(\cdot,\alpha)$ and then let us consider $n_0>\max\{z_x,z_y\}$.\\
	We will equivalently prove that the series
	\begin{equation}\label{eq:seriestrunc}
	\sum_{n=n_0}^{+\infty}\frac{1}{d_n^2}E_\nu(-bnt^\nu)C_x(n,\alpha)C_y(n,\alpha)
	\end{equation}
	converges. To do this, we need to recall another property of the Charlier polynomials. In particular it is known (see \cite[Table $2.3$]{NikiforovUvarovSuslov1991}) that the director coefficient of $C_n(\cdot,\alpha)$ is given by
	\begin{equation*}
	c_n=\frac{1}{(-\alpha)^n}.
	\end{equation*}
	In particular, recalling that $\alpha=\frac{a}{b}$, $\alpha>0$ since $a,b>0$ and then $c_n>0$ if $n$ is even and $c_n<0$ if $n$ is odd.\\
	By using this observation, we can distinguish two cases:
	\begin{itemize}
		\item[i] If $x+y$ is even, then, since $c_xc_y>0$, for any $n \ge n_0$ $C_x(n,\alpha)C_y(n,\alpha)>0$ and then the series \eqref{eq:seriestrunc} admits only positive summands. Recalling that $E_\nu(-bnt^\nu)\le 1$ we obtain
		\begin{equation*}
		\sum_{n=n_0}^{+\infty}\frac{1}{d_n^2}E_\nu(-bnt^\nu)C_x(n,\alpha)C_y(n,\alpha)\le\sum_{n=n_0}^{+\infty}\frac{1}{d_n^2}C_x(n,\alpha)C_y(n,\alpha)
		\end{equation*}
		where the RHS series converges since
		\begin{equation}\label{eq:seriesconv}
		\sum_{n=0}^{+\infty}\frac{1}{d_n^2}C_x(n,\alpha)C_y(n,\alpha)=e^{-\alpha}\sum_{n=0}^{+\infty}\frac{\alpha^n}{n!}e^{\alpha}C_x(n,\alpha)C_y(n,\alpha)=e^{\alpha}d_{x}^2\delta_{x,y}.
		\end{equation}
		\item[ii] If $x+y$ is odd, then, since $c_xc_y<0$, for any $n \ge n_0$ $C_x(n,\alpha)C_y(n,\alpha)<0$ and then the series \eqref{eq:seriestrunc} admits only negative summands. As before, we obtain
		\begin{equation*}
		\sum_{n=n_0}^{+\infty}\frac{1}{d_n^2}E_\nu(-bnt^\nu)C_x(n,\alpha)C_y(n,\alpha)\ge\sum_{n=n_0}^{+\infty}\frac{1}{d_n^2}C_x(n,\alpha)C_y(n,\alpha)
		\end{equation*}
		where the RHS series converges for equation \eqref{eq:seriesconv}.
	\end{itemize}
\end{proof}
With Lemma \ref{lem:fundsol}, we have exploited the fundamental solution of the equation in \eqref{eq:prob}. Now we have to show that a function in the form \eqref{eq:specdecheur} is a solution for such fractional Cauchy problem. To do this, let us first show a technical lemma.
\begin{lem}\label{lem:teclemma}
	For any $t_0>0$, there exists a constant $K(t_0,\nu)$ such that
	\begin{equation*}
	bnE_\nu(-bnt^\nu)\le K(t_0,\nu), \ t \in [t_0,+\infty).
	\end{equation*}
\end{lem}
\begin{proof}
	Let us use the uniform estimate for the Mittag-Leffler function given in \cite[Theorem $4$]{Simon2014}:
	\begin{equation*}
	bnE_\nu(-bnt^{\nu})\le \frac{bn}{1+\frac{bnt^{\nu}}{\Gamma(1+\nu)}}.
	\end{equation*}
	Consider the function
	\begin{equation*}
	f(x)=\frac{x}{1+Cx}, \quad C=\frac{t^\nu}{\Gamma(1+\nu)}.
	\end{equation*}
	Thus we have
	\begin{equation*}
	f'(x)=\frac{1}{(1+Cx)^2}>0
	\end{equation*}
	hence the function $f$ is strictly increasing. So we have
	\begin{equation*}
	f(x)\le \lim_{x \to +\infty}f(x)=\frac{1}{C}=\frac{\Gamma(1+\nu)}{t^\nu},
	\end{equation*}
	and then
	\begin{equation*}
	bnE_\nu(-bnt^{\nu})\le \frac{bn}{1+\frac{bnt^{\nu}}{\Gamma(1+\nu)}}\le \frac{\Gamma(1+\nu)}{t^\nu}\le \frac{\Gamma(1+\nu)}{t_0^\nu}=:K(t_0,\nu).
	\end{equation*}
\end{proof}
Now let us exhibit a strong solution for our fractional Cauchy problem.
\begin{thm}\label{thm:strongback}
	Let $g \in \ell^2(\mm)$ with decomposition $g(x)=\sum_{n=0}^{+\infty}g_nQ_n(x)$. Then the fractional difference-differential Cauchy problem
	\begin{equation}\label{eq:backprob}
	\begin{cases}
	\frac{\partial^\nu u}{\partial t^\nu}(t,x)=\cG u(t,x) \\
	u(0,x)=g(x),
	\end{cases}
	\end{equation}
	admits a strong solution $u$ in the form
	\begin{equation}\label{eq:strongsol1}
	u(t,x)=\sum_{y=0}^{+\infty}p_\nu(t,y;x)g(y)=\sum_{n=0}^{\infty}E_\nu(-bnt^\nu)Q_n(x)g_n.
	\end{equation}
\end{thm}
\begin{proof}
	First let us observe that obviously if $u$ is in the form \eqref{eq:strongsol1}, then $u(0,x)=g(x)$.\\
	Now, let us notice that
	\begin{multline*}
	\cG E_\nu(-bnt^\nu)Q_n(x)g_n=E_\nu(-bnt^\nu)g_n\cG Q_n(x)\\=-bnE_\nu(-bnt^\nu)g_nQ_n(x)=g_nQ_n(x)\derfrac{E_\nu(-bnt^\nu)}{t}.
	\end{multline*}
	Hence we need to show that the series in \eqref{eq:strongsol1} is convergent at least uniformly in $t$ and that we can change the series with the operators.\\
	Starting from the convergence of the series, by using Cauchy-Schwartz inequality we have
	\begin{align}\label{eq:432}
	\begin{split}
	\sum_{n=0}^{+\infty}|E_\nu(-bnt^\nu)Q_n(x)g_n|&\le\sum_{n=0}^{+\infty}|Q_n(x)g_n|\\&\le \left(\sum_{n=0}^{+\infty}\frac{\alpha^n}{n!}C_n^2(x,\alpha)\right)^{\frac{1}{2}}\left(\sum_{n=0}^{+\infty}g_n^2\right)^{\frac{1}{2}}\\&=\Norm{g}{\ell^2(\mm)}\left(\sum_{n=0}^{+\infty}\frac{\alpha^n}{n!}C_x^2(n,\alpha)\right)^{\frac{1}{2}}\\&=\Norm{g}{\ell^2(\mm)}e^{\frac{\alpha}{2}}d_x,
	\end{split}
	\end{align}
	hence the series in \eqref{eq:strongsol1} totally converges.\\
	Now we need to show that one can exchange the operators with the series. To do that, let us first observe that
	\begin{equation*}
	\int_0^t(t-\tau)^{-\nu}u(\tau)d\tau=\int_0^t\frac{u(\tau)}{\nu-1}d(t-\tau)^{1-\nu},
	\end{equation*}
	and since $(t-\tau)^{1-\nu}$ is strictly decreasing in $[0,t]$ we can use \cite[Theorem $7.16$]{Rudin1976} with the total convergence of the series \eqref{eq:strongsol1} to obtain
	\begin{equation*}
	\int_0^t(t-\tau)^{-\nu}\sum_{n=0}^{+\infty}E_\nu(-bn\tau^\nu)Q_n(x)g_nd\tau=\sum_{n=0}^{+\infty}\int_0^t(t-\tau)^{-\nu}E_\nu(-bn\tau^\nu)Q_n(x)g_nd\tau.
	\end{equation*}
	Now we want to use the following relation:
	\begin{align*}
	\der{}{t}\int_0^t(t-\tau)^{-\nu}\sum_{n=0}^{+\infty}E_\nu(-bn\tau^\nu)Q_n(x)g_nd\tau&=\der{}{t}\sum_{n=0}^{+\infty}\int_0^t(t-\tau)^{-\nu}E_\nu(-bn\tau^\nu)Q_n(x)g_nd\tau\\&=\sum_{n=0}^{+\infty}\der{}{t}\int_0^t(t-\tau)^{-\nu}E_\nu(-bn\tau^\nu)Q_n(x)g_nd\tau,
	\end{align*}
	but to do this, by using \cite[Theorem $7.17$]{Rudin1976}, we need to show the uniform convergence of
	\begin{equation*}
	\sum_{n=0}^{+\infty}\der{}{t}\int_0^t(t-\tau)^{-\nu}E_\nu(-bn\tau^\nu)Q_n(x)g_nd\tau,
	\end{equation*}
	in any compact interval included in $(0,+\infty)$. Hence, by definition of Caputo fractional derivative, as given in \eqref{eq:caputoder}, we really need to show the uniform convergence of
	\begin{equation}\label{eq:fracderserie}
	\sum_{n=0}^{+\infty}\derfrac{}{t}E_\nu(-bnt^\nu)Q_n(x)g_n
	\end{equation}
	in any interval of the form $[t_0,+\infty)$. To do this, let us recall that $\derfrac{}{t}E_\nu(-bnt^\nu)=-bnE_\nu(-bnt^\nu)$ and thus we need to show the uniform convergence of
	\begin{equation*}
	\sum_{n=0}^{+\infty}-bnE_\nu(-bnt^\nu)Q_n(x)g_n.
	\end{equation*}
	Thus, fix $t_0>0$ and observe that
	\begin{align*}
	\begin{split}
	\sum_{n=0}^{+\infty}|bnE_\nu(-bnt^\nu)Q_n(x)g_n|&\le K(t_0,\nu)\sum_{n=0}^{+\infty}|Q_n(x)g_n|\\&\le K(t_0,\nu)\Norm{g}{\ell^2(\mm)}e^{\frac{\alpha}{2}}d_x, \quad t \in [t_0,+\infty),
	\end{split}
	\end{align*}
	where the first inequality follows from Lemma \ref{lem:teclemma} and the second inequality from Cauchy-Schwartz inequality as done before in \eqref{eq:432}. Hence we have shown the total convergence of \eqref{eq:fracderserie} in any interval of the form $[t_0,+\infty)$.\\
	We have already shown that $\sum_{n=0}^{+\infty}E_\nu(-bnt^\nu)Q_n(x)g_n$ totally converges with respect to $t$: in the same way we have that also $\sum_{n=0}^{+\infty}E_\nu(-bnt^\nu)Q_n(x-1)g_n$ and $\sum_{n=0}^{+\infty}E_\nu(-bnt^\nu)Q_n(x+1)g_n$ totally converge with respect to $t$.\\
	Now, observe that
	\begin{align}
	\begin{split}
	\nabla^-\sum_{n=0}^{+\infty}&E_\nu(-bnt^\nu)Q_n(x)g_n=\sum_{n=0}^{+\infty}E_\nu(-bnt^\nu)Q_n(x)g_n-\sum_{n=0}^{+\infty}E_\nu(-bnt^\nu)Q_n(x-1)g_n\\&=\lim_{N\to+\infty}\sum_{n=0}^{N}E_\nu(-bnt^\nu)Q_n(x)g_n-\lim_{N \to +\infty}\sum_{n=0}^{N}E_\nu(-bnt^\nu)Q_n(x)g_n\\&=\lim_{N \to +\infty}\left(\sum_{n=0}^{N}E_\nu(-bnt^\nu)Q_n(x)g_n-\sum_{n=0}^{N}E_\nu(-bnt^\nu)Q_n(x-1)g_n\right)\\&=\lim_{N \to +\infty}\sum_{n=0}^{N}E_\nu(-bnt^\nu)\nabla^-Q_n(x)g_n\\&=\sum_{n=0}^{+\infty}E_\nu(-bnt^\nu)\nabla^-Q_n(x)g_n,
	\end{split}
	\end{align}
	and in the same way one can show that
	\begin{equation*}
	\Delta \sum_{n=0}^{+\infty}E_\nu(-bnt^\nu)Q_n(x)g_n=\sum_{n=0}^{+\infty}E_\nu(-bnt^\nu)\Delta Q_n(x)g_n.
	\end{equation*}
	By using these last two relations, it is easy to show that
	\begin{equation*}
	\cG\sum_{n=0}^{+\infty}E_\nu(-bnt^\nu)Q_n(x)g_n=\sum_{n=0}^{+\infty}E_\nu(-bnt^\nu)\cG Q_n(x)g_n.
	\end{equation*}
	Finally we have that
	\begin{align*}
	\begin{split}
	\derfrac{}{t}\sum_{n=0}^{+\infty}E_\nu(-bnt^\nu)Q_n(x)g_n&=\sum_{n=0}^{+\infty}\derfrac{}{t}E_\nu(-bnt^\nu)Q_n(x)g_n\\&=\sum_{n=0}^{+\infty}\cG E_\nu(-bnt^\nu)Q_n(x)g_n\\&=\cG\sum_{n=0}^{+\infty}E_\nu(-bnt^\nu)Q_n(x)g_n,
	\end{split}
	\end{align*}
	and we have concluded the proof.
\end{proof}
The same strategy can be used to exhibit a strong solution to the fractional forward Kolmogorov equation.
\begin{thm}\label{thm:strongfor}
	Let $f$ be a function such that $f/m \in \ell^2(\mm)$ with decomposition {$f(x)/m(x)=\sum_{n=0}^{+\infty}f_nQ_n(x)$}. Then the fractional difference-differential Cauchy problem
	\begin{equation}\label{eq:forprob}
	\begin{cases}
	\frac{\partial^\nu u}{\partial t^\nu}(t,x)=\cL u(t,x)\\
	u(0,x)=f(x),
	\end{cases}
	\end{equation}
	admits a strong solution $u=u(t,x)$ given by
	\begin{equation*}
	u(t,x)=\sum_{y=0}^{+\infty}p_\nu(t,x;y)f(y)=m(x)\sum_{n=0}^{+\infty}E_\nu(-bnt^\nu)Q_n(x)f_n.
	\end{equation*}
\end{thm}
\begin{proof}
	Since $\{f_{n}\}_{n \in \mathbb{N}}\in \ell^2$, then, from the previous theorem, we already know that we can exchange operators and series. We only need to prove that the single summand of the series is a solution of the equation and that $u(0,x)=f(x)$.\\
	Let us first notice that
	\begin{equation*}
	u(0,x)=m(x)\sum_{n=0}^{+\infty}Q_n(x)f_n=m(x)\frac{f(x)}{m(x)}=f(x),
	\end{equation*}
	thus the function $u$ satisfies the given initial condition.\\
	To show that the single summand is solution of the equation, let us write $\cL$ as
	\begin{equation*}
	\cL h(x)=-\nabla^-((a-bz)h(z))(x)+\Delta(bzh(z))(x),
	\end{equation*}
	for a generic function $h$.\\
	Moreover, let us observe that
	\begin{align*}
	\begin{split}
	\cG h(x)&=(a-bx)\nabla^-h(x)+a\Delta h(x)\\
	&=ah(x)-ah(x-1)-bx\nabla^-h(x)+ah(x+1)-2ah(x)+ah(x-1)\\
	&=ah(x+1)-ah(x)-bx\nabla^-h(x)\\
	&=a\nabla^+h(x)-bx\nabla^-h(x).
	\end{split}
	\end{align*}
	Let us also recall that $m$ solves a discrete Pearson equation:
	\begin{equation}\label{eq:discPear}
	\nabla^+(b\cdot m(z))(x)=(a-bx)m(x).
	\end{equation}
	Now, let us observe that
	\begin{equation*}
	\cL(m(z)Q_n(z)E_\nu(-bnt^\nu)f_n)(x)=f_nE_\nu(-bnt^\nu)\cL(m(z)Q_n(z))(x),
	\end{equation*}
	hence we will only study $\cL(m(\cdot)Q_n(\cdot))$. In particular we have
	\begin{equation*}
	\cL(m(z)Q_n(z))(x)=-\nabla^-((a-bz)m(z)Q_n(z))(x)+\Delta(bz m(z)Q_n(z))(x),
	\end{equation*}
	hence, by using the Discrete Leibnitz Rule (\eqref{eq:backLeibrule} and \eqref{eq:LapLeibrule}), we obtain
	\begin{align*}
	\begin{split}
	\cL(m(z)Q_n(z))(x)&=-[Q_n(x)\nabla^-((a-bz)m(z))(x)+(a-b(x-1))m(x-1)\nabla^-Q_n(x)]\\&+Q_n(x)\Delta(bz m(z))(x)+b(x+1)m(x+1)\nabla^+Q(x)+\\&-b(x-1)m(x-1)\nabla^-Q_n(x)\\
	&=Q_n(x)[-\nabla^-((a-bz)m(z))(x)+\Delta(bz m(z))(x)]+\\
	&-am(x-1)\nabla^-Q_n(x)+b(x+1)m(x+1)\nabla^+Q_n(x).
	\end{split}
	\end{align*}
	First let us observe that $\Delta=\nabla^-\nabla^+$, then
	\begin{align*}
	\begin{split}
	-\nabla^-((a-bz)m(z))+\Delta(bz m(z))&=-\nabla^-((a-bz)m(z))+\nabla^-\nabla^+(bz m(z))\\&=\nabla^-(\nabla^+(bz m(z))(x)-(a-bx)m(x))=0,
	\end{split}
	\end{align*}
	since $m$ satisfies equation \eqref{eq:discPear}. Moreover
	\begin{equation*}
	am(x-1)=a\frac{\alpha^{x-1}}{(x-1)!}e^{-\alpha}=\frac{a}{\alpha}xm(x)=bxm(x),
	\end{equation*}
	while
	\begin{equation*}
	b(x+1)m(x+1)=b(x+1)\frac{\alpha^{x+1}}{(x+1)!}e^{-\alpha}=b\alpha m(x)=am(x),
	\end{equation*}
	thus
	\begin{align*}
	\begin{split}
	\cL(m(z)Q_n(z))(x)&=-bxm(x)\nabla^-Q_n(x)+am(x)\nabla^+Q_n(x)\\&=m(x)[\nabla^+Q_n(x)-bx\nabla^-Q_n(x)]\\&=m(x)\cG Q_n(x).
	\end{split}
	\end{align*}
	Finally, we obtain:
	\begin{align*}
	\begin{split}
	\cL(m(z)Q_n(z)E_\nu(-bnt^\nu)f_n)&=f_nE_\nu(-bnt^\nu)\cL(m(z)Q_n(z))\\&=f_nE_\nu(-bnt^\nu)m(x)\cG Q_n(x)\\&=-bnf_nE_\nu(-bnt^\nu)m(x)Q_n(x)\\&=f_nm(x)Q_n(x)\derfrac{E_\nu(-bnt^\nu)}{t}\\&=\derfrac{}{t}(f_nm(x)Q_n(x)E_\nu(-bnt^\nu)).
	\end{split}
	\end{align*}
\end{proof}
\begin{rmk}
	It is easy to see that $p_\nu(t,x;y)$ is strong solution of the fractional backward equation
	\begin{equation}\label{eq:fundback}
	\begin{cases}
	\derfrac{p}{t}(t,x;y)=\cG p(t,x;y)\\
	p_\nu(0,x;y)=\delta_{x,y},
	\end{cases}
	\end{equation}
	where $\cG$ operates on $y$, and is also strong solution of the fractional forward equation
	\begin{equation}\label{eq:fundfor}
	\begin{cases}
	\derfrac{p}{t}(t,x;y)=\cL p(t,x;y)\\
	p_\nu(0,x;y)=\delta_{x,y},
	\end{cases}
	\end{equation}
	where $\cL$ operates on $x$. In particular, as shown by Theorems \ref{thm:strongback} and \ref{thm:strongfor}, it is the fundamental solution of such equations.
\end{rmk}
\section{Stochastic representation of the solutions}\label{sec5}
Now we want to exhibit a process whose ``transition probability'' is the fundamental solution $p_\nu(t,x;y)$ we have described previously.\\
To do this, let us consider a classical immigration-death process $N_1(t)$ (as defined before). Let us also consider a $\nu$-stable subordinator $\sigma_\nu(t)$ with Laplace transform
\begin{equation*}
\E[e^{-s\sigma_\nu(t)}]=e^{-ts^\nu}, \ s>0, \ \nu \in (0,1)
\end{equation*}
and its inverse process (or first passage time process) $L_\nu(t)$ defined as
\begin{equation*}
L_\nu(t):=\inf\{s>0: \ \sigma_\nu(s)>t\}.
\end{equation*}
The latter admits density (see \cite{Aletti2018,MeerschaertStraka2013})
\begin{equation*}
\bP(L_\nu(t) \in dy)=f_\nu(y,t)dy=\frac{t}{\nu}\frac{1}{y^{1+\frac{1}{\nu}}}g_\nu\left(\frac{t}{y^{\frac{1}{\nu}}}\right)dy \ y \ge 0, \ t>0,
\end{equation*}
where $g_\nu(x)$ is the density of $\sigma_\nu(1)$ given by 
\begin{equation*}
g_\nu(x)=\frac{1}{\pi}\sum_{k=1}^{+\infty}(-1)^{k+1}\frac{\Gamma(\nu k+1)}{k!}\frac{1}{x^{\nu k+1}}, \ x \ge 0.
\end{equation*}
Alternatives for $f_\nu(y,t)$ are given in \cite{Kataria2018,Kumar2015}.\\
Thus, let us define the \textit{fractional immigration-death process} as $N_\nu(t):=N_1(L_\nu(t))$. This is a semi-Markov process as defined in \cite{GihmanSkorohod11975}. However, we say that such process admits a transition probability mass $p_\nu(t,x;y)$ if for any $B \subseteq \mathbb{N}_0$:
\begin{equation*}
\bP(N_\nu(t)\in B| \ N_\nu(0)=y)=\sum_{x \in B}p_\nu(t,x;y).
\end{equation*}
Hence, we can use such process to characterize the fundamental solution we found in the previous section.
\begin{thm}\label{thm:pnu}
	The process $N_\nu(t)$ admits a transition probability mass $p_\nu(t,x;y)$ in the form \eqref{eq:fundsol}.
\end{thm}
\begin{proof}
	Let us first recall that (see, for instance, \cite{MeerschaertStraka2013}) the process $L_\nu(t)$ admits a density $f_t(\tau)=\bP(L_\nu(t)\in d\tau)$. Moreover, let us recall (see \cite{Bingham1971}) that
	\begin{equation*}
	\int_0^{+\infty}e^{-s\tau}f_t(\tau)d\tau=E_\nu(-st^\nu), \ s>0.
	\end{equation*}
	Now, observe that for any $B \subseteq \mathbb{N}_0$, since $N_1(t)$ admits a transition probability mass, we have
	\begin{align*}
	\begin{split}
	\bP(N_\nu(t)\in B|N_\nu(0)=y)&=\int_0^{+\infty}\bP(N_1(\tau)\in B|N_1(0)=y)f_t(\tau)d\tau\\
	&=\int_0^{+\infty}\sum_{x \in B}p_1(\tau,x;y)f_t(\tau)d\tau.
	\end{split}
	\end{align*}
	Now, if $B$ is a finite set, we have
	\begin{equation*}
	\int_0^{+\infty}\sum_{x \in B}p_1(\tau,x;y)f_t(\tau)d\tau=\sum_{x \in B}\int_0^{+\infty}p_1(\tau,x;y)f_t(\tau)d\tau.
	\end{equation*}
	If $B$ is infinite, let us consider the sets $I_m:=\{x \in \mathbb{N}_0: \ x \le m\}$ and $B_m:=B \cap I_m$. Thus, $B_m$ is finite and then
	\begin{equation*}
	\int_0^{+\infty}\sum_{x \in B_m}p_1(\tau,x;y)f_t(\tau)d\tau=\sum_{x \in B_m}\int_0^{+\infty}p_1(\tau,x;y)f_t(\tau)d\tau.
	\end{equation*}
	Since $p_1(\tau,x;y)f_t(\tau)$ is non-negative, we can use the monotone convergence theorem to obtain, taking the limit as $m \to +\infty$
	\begin{equation*}
	\int_0^{+\infty}\sum_{x \in B}p_1(\tau,x;y)f_t(\tau)d\tau=\sum_{x \in B}\int_0^{+\infty}p_1(\tau,x;y)f_t(\tau)d\tau.
	\end{equation*}
	Now let us only consider
	\begin{equation*}
	\int_0^{+\infty}p_1(\tau,x;y)f_t(\tau)d\tau,
	\end{equation*}
	and recall that (see, for instance, \cite{Kuznetsov2004} for the specific case of the immigration-death process, but in general such decomposition can be found in \cite{KarlinMcGregor1957a,KarlinMcGregor1957b})
	\begin{equation*}
	p_1(\tau,x;y)=m(x)\sum_{n=0}^{+\infty}e^{-bn\tau}Q_n(x)Q_n(y).
	\end{equation*}
	Hence we have
	\begin{equation*}
	\int_0^{+\infty}p_1(\tau,x;y)f_t(\tau)d\tau=m(x)\int_0^{+\infty}\sum_{n=0}^{+\infty}e^{-bn\tau}Q_n(x)Q_n(y)f_t(\tau)d\tau.
	\end{equation*}
	Now we have to show that we can exchange integral and series. To do this, let us first observe that
	\begin{align*}
	\begin{split}
	\sum_{n=0}^{+\infty}e^{-bn\tau}Q_n(x)Q_n(y)f_t(\tau)&=\sum_{n=0}^{+\infty}\frac{\alpha^n}{n!}e^{-bn\tau}C_n(x,\alpha)C_n(y,\alpha)f_t(\tau)\\&=\sum_{n=0}^{+\infty}\frac{\alpha^n}{n!}e^{-bn\tau}C_x(n,\alpha)C_y(n,\alpha)f_t(\tau).
	\end{split}
	\end{align*}
	Let us consider $z_x$ and $z_y$ the last real zeros of $C_x(n,\alpha)$ and $C_y(n,\alpha)$ and consider a $n_0 \in \mathbb{N}$ such that $n_0>\max\{z_x,z_y\}$. Thus we have
	\begin{align*}
	\begin{split}
	\sum_{n=0}^{+\infty}\frac{\alpha^n}{n!}e^{-bn\tau}C_x(n,\alpha)C_y(n,\alpha)f_t(\tau)&=\sum_{n=0}^{n_0}\frac{\alpha^n}{n!}e^{-bn\tau}C_x(n,\alpha)C_y(n,\alpha)f_t(\tau)\\&+\sum_{n=n_0+1}^{+\infty}\frac{\alpha^n}{n!}e^{-bn\tau}C_x(n,\alpha)C_y(n,\alpha)f_t(\tau),
	\end{split}
	\end{align*}
	and
	\begin{align*}
	\begin{split}
	\int_0^{+\infty}\sum_{n=0}^{+\infty}\frac{\alpha^n}{n!}e^{-bn\tau}C_x(n,\alpha)C_y(n,\alpha)f_t(\tau)d\tau&=\int_0^{+\infty}\sum_{n=0}^{n_0}\frac{\alpha^n}{n!}e^{-bn\tau}C_x(n,\alpha)C_y(n,\alpha)f_t(\tau)d\tau\\&+\int_0^{+\infty}\sum_{n=n_0+1}^{+\infty}\frac{\alpha^n}{n!}e^{-bn\tau}C_x(n,\alpha)C_y(n,\alpha)f_t(\tau)d\tau\\&=\sum_{n=0}^{n_0}\frac{\alpha^n}{n!}C_x(n,\alpha)C_y(n,\alpha)\int_0^{+\infty}e^{-bn\tau}f_t(\tau)d\tau\\&+\int_0^{+\infty}\sum_{n=n_0+1}^{+\infty}\frac{\alpha^n}{n!}e^{-bn\tau}C_x(n,\alpha)C_y(n,\alpha)f_t(\tau)d\tau.
	\end{split}
	\end{align*}
	Now, fix $\tau_0>0$ and observe that for $\tau>\tau_0$ and $n \ge n_0+1$ the function 
	\begin{equation*}
	(\tau,n) \mapsto \frac{\alpha^n}{n!}e^{-bn\tau}C_x(n,\alpha)C_y(n,\alpha)f_t(\tau)
	\end{equation*}
	does not change sign, by Fubini's theorem (see \cite[Theorem $8.8$]{Rudin2006}) we have that
	\begin{equation*}
	\int_{\tau_0}^{+\infty}\sum_{n=n_0+1}^{+\infty}\frac{\alpha^n}{n!}e^{-bn\tau}C_x(n,\alpha)C_y(n,\alpha)f_t(\tau)d\tau=\sum_{n=n_0+1}^{+\infty}\frac{\alpha^n}{n!}C_x(n,\alpha)C_y(n,\alpha)\int_{\tau_0}^{+\infty}e^{-bn\tau}f_t(\tau)d\tau.
	\end{equation*}
	Now we have to pass to the limit as $\tau_0 \to 0$. To do this, let us observe that
	\begin{equation*}
	\int_{\tau_0}^{+\infty}e^{-bn\tau}f_t(\tau)d\tau\le \int_{0}^{+\infty}e^{-bn\tau}f_t(\tau)d\tau=E_\nu(-bnt^\nu),
	\end{equation*}
	and let us distinguish two cases.
	\begin{itemize}
		\item[$i)$] If $x+y$ is even, 
		\begin{equation*}
		\frac{\alpha^n}{n!}C_x(n,\alpha)C_y(n,\alpha)\int_{\tau_0}^{+\infty}e^{-bn\tau}f_t(\tau)d\tau \ge 0,
		\end{equation*}
		and in particular we have
		\begin{align*}
		\begin{split} 
		\frac{\alpha^n}{n!}C_x(n,\alpha)C_y(n,\alpha)\int_{\tau_0}^{+\infty}e^{-bn\tau}f_t(\tau)d\tau &\le \frac{\alpha^n}{n!}C_x(n,\alpha)C_y(n,\alpha)E_\nu(-bnt^\nu)\\
		&\le \frac{\alpha^n}{n!}C_x(n,\alpha)C_y(n,\alpha),
		\end{split}
		\end{align*}
		where
		\begin{equation*}
		\sum_{n=n_0+1}^{+\infty}\frac{\alpha^n}{n!}C_x(n,\alpha)C_y(n,\alpha)<+\infty,
		\end{equation*}
		as we observed before. Then we can use dominated convergence theorem to take the limit as $\tau_0 \to 0$ and obtain
		\begin{equation*}
		\int_0^{+\infty}\sum_{n=n_0+1}^{+\infty} C_x(n,\alpha)C_y(n,\alpha)e^{-bn\tau}f_t(\tau)d\tau=\sum_{n=n_0+1}^{+\infty} C_x(n,\alpha)C_y(n,\alpha)\int_0^{+\infty}e^{-bn\tau}f_t(\tau)d\tau.
		\end{equation*}
		\item[$ii)$] If $x+y$ is odd, then
		\begin{equation*}
		\frac{\alpha^n}{n!}C_x(n,\alpha)C_y(n,\alpha)\int_{\tau_0}^{+\infty}e^{-bn\tau}f_t(\tau)d\tau \le 0,
		\end{equation*}
		and in particular we have
		\begin{align*}
		\begin{split} 
		-\frac{\alpha^n}{n!}C_x(n,\alpha)C_y(n,\alpha)\int_{\tau_0}^{+\infty}e^{-bn\tau}f_t(\tau)d\tau &\le -\frac{\alpha^n}{n!}C_x(n,\alpha)C_y(n,\alpha)E_\nu(-bnt^\nu)\\
		&\le -\frac{\alpha^n}{n!}C_x(n,\alpha)C_y(n,\alpha),
		\end{split}
		\end{align*}
		where
		\begin{equation*}
		-\sum_{n=n_0+1}^{+\infty}\frac{\alpha^n}{n!}C_x(n,\alpha)C_y(n,\alpha)<+\infty,
		\end{equation*}
		as we observed before. Then we can use dominated convergence theorem to take the limit as $\tau_0 \to 0$ and obtain
		\begin{equation*}
		\int_0^{+\infty}\sum_{n=n_0+1}^{+\infty} C_x(n,\alpha)C_y(n,\alpha)e^{-bn\tau}f_t(\tau)d\tau=\sum_{n=n_0+1}^{+\infty} C_x(n,\alpha)C_y(n,\alpha)\int_0^{+\infty}e^{-bn\tau}f_t(\tau)d\tau.
		\end{equation*}
	\end{itemize}
	Hence in general we have for any $x,y \in \mathbb{N}_0$
	\begin{equation*}
	\int_0^{+\infty}\sum_{n=n_0+1}^{+\infty} C_x(n,\alpha)C_y(n,\alpha)e^{-bn\tau}f_t(\tau)d\tau=\sum_{n=n_0+1}^{+\infty} C_x(n,\alpha)C_y(n,\alpha)\int_0^{+\infty}e^{-bn\tau}f_t(\tau)d\tau
	\end{equation*}
	and then
	\begin{align*}
	\begin{split}
	\int_0^{+\infty}p_1(\tau,x;y)f_t(\tau)d\tau&=\int_0^{+\infty}\sum_{n=0}^{+\infty}\frac{\alpha^n}{n!}e^{-bn\tau}C_x(n,\alpha)C_y(n,\alpha)f_t(\tau)d\tau\\&=m(x)\sum_{n=0}^{n_0}\frac{\alpha^n}{n!}C_x(n,\alpha)C_y(n,\alpha)\int_0^{+\infty}e^{-bn\tau}f_t(\tau)d\tau\\&+m(x)\int_0^{+\infty}\sum_{n=n_0+1}^{+\infty}\frac{\alpha^n}{n!}e^{-bn\tau}C_x(n,\alpha)C_y(n,\alpha)f_t(\tau)d\tau\\&=m(x)\sum_{n=0}^{n_0}\frac{\alpha^n}{n!}C_x(n,\alpha)C_y(n,\alpha)\int_0^{+\infty}e^{-bn\tau}f_t(\tau)d\tau\\&+m(x)\sum_{n=n_0+1}^{+\infty}\frac{\alpha^n}{n!}C_x(n,\alpha)C_y(n,\alpha)\int_0^{+\infty}e^{-bn\tau}f_t(\tau)d\tau\\&=m(x)\sum_{n=0}^{+\infty}\frac{\alpha^n}{n!}C_x(n,\alpha)C_y(n,\alpha)\int_0^{+\infty}e^{-bn\tau}f_t(\tau)d\tau\\&=m(x)\sum_{n=0}^{+\infty}\frac{\alpha^n}{n!}C_x(n,\alpha)C_y(n,\alpha)E_\nu(-bnt^\nu).
	\end{split}
	\end{align*}
	Finally we have
	\begin{align*}
	\begin{split}
	\bP(N_\nu(t) \in B|N_\nu(0)=y)&=\int_0^{+\infty}\sum_{x \in B}p_1(\tau,x;y)f_t(\tau)d\tau\\&=\sum_{x \in B}\int_0^{+\infty}p_1(\tau,x;y)f_t(\tau)d\tau\\&=\sum_{x \in B}m(x)\sum_{n=0}^{+\infty}\frac{\alpha^n}{n!}C_x(n,\alpha)C_y(n,\alpha)E_\nu(-bnt^\nu).
	\end{split}
	\end{align*}
	Thus $p_\nu(t,x;y)$ exists and
	\begin{equation*}
	p_\nu(t,x;y)=m(x)\sum_{n=0}^{+\infty}\frac{\alpha^n}{n!}C_x(n,\alpha)C_y(n,\alpha)E_\nu(-bnt^\nu).
	\end{equation*}
\end{proof}
Now we are ready to prove the following Theorem.
\begin{thm}\label{thm:stocrepback}
	Let $g \in c_0$ such that $g(x)=\sum_{n=0}^{+\infty}g_nQ_n(x)$. Then the function
	\begin{equation*}
	u(t;x)=\E[g(N_\nu(t))|N_\nu(0)=x]
	\end{equation*}
	is a solution of \eqref{eq:backprob}.
\end{thm}
\begin{proof}
	First of all observe that
	\begin{equation*}
	\sum_{x=0}^{+\infty}m(x)g^2(x)\le \Norm{g}{\ell^\infty}^2,
	\end{equation*}
	hence $g \in \ell^2(\mm)$ and we are under the hypotheses of Theorem \ref{thm:strongback}.\\
	Consider a generic $f \in c_0$ and define the family of operators
	\begin{equation*}
	T_tf(x)=\E[f(N_1(t))|N_1(0)=x]=\sum_{y \in \mathbb{N}_0}p_1(t,y;x)f(y).
	\end{equation*}
	In particular, by Lemma \ref{lem:Fellemma}, we know that $N_1(t)$ is a Feller process, hence $(T_t)_{t \ge 0}$ is a Feller semigroup. Moreover, strong continuity of $(T_t)_{t \ge 0}$ follows from \cite[Lemma $1.4$]{Bott2013}. Then, by using \cite[Theorem $3.1$]{BaeumerMeerschaert}, we know that, since $\cG$ is the generator of $T_t$, the function
	\begin{equation*}
	u(t;x):=\int_{0}^{+\infty}T_{\left(\frac{t}{s}\right)^\nu}g(x)g_\nu(s)ds,
	\end{equation*}
	where $g_\nu(s)$ is the density of $\sigma_\nu(1)$, is a solution of \eqref{eq:backprob}. But if we use the change of variables $\tau=\left(\frac{t}{s}\right)^\nu$, and the fact that $f_t(\tau)=\frac{t}{\nu}\tau^{-1-\frac{1}{\nu}}g_\nu(t\tau^{-\frac{1}{\nu}})$ for $\tau \ge 0$ (see, for instance, \cite{MeerschaertStraka2013}), we obtain
	\begin{align*}
	\begin{split}
	u(t;x)&=\frac{t}{\nu}\int_0^{\infty}T_\tau g(x)\tau^{-1-\frac{1}{\nu}}g_\nu(t\tau^{-\frac{1}{\nu}})d\tau\\&=\int_0^{+\infty}T_\tau g(x)f_t(\tau)d\tau\\&=\int_0^{+\infty}\E[g(N_1(\tau))|N_1(0)=x]f_t(\tau)d\tau\\&=\E[g(N_1(L_\nu(\tau)))|N_1(0)=x]=\E[g(N_\nu(t))|N_\nu(0)=x], \ t\ge 0, \ x \in \mathbb{N}.
	\end{split}
	\end{align*}
\end{proof}
Finally, we can provide the stochastic representation of solutions of \eqref{eq:forprob}.
\begin{cor}\label{cor:stocrap}
	Let $p_\nu(t,x;y)$ be the transition density of $N_\nu(t)$. Then, for any $f$ such that $\frac{f}{m}\in \ell^2(\mm)$ with decomposition $f(x)/m(x)=\sum_{n=0}^{+\infty}f_nQ_n(x)$. Thus
	\begin{equation*}
	u(t,x)=\sum_{y \in \mathbb{N}_0}p_\nu(t,x;y)f(y)
	\end{equation*}
	is a solution of \eqref{eq:forprob}.
\end{cor}
\begin{proof}
	This easily follows from Theorems \ref{thm:pnu} and \ref{thm:strongfor}.
\end{proof} 
We can use the last Corollary to exploit the asymptotic behaviour of the density of the process $N_\nu(t)$.
\section{Uniqueness of strong solutions}\label{sec6}
In this section, we aim to show that the strong solutions of \eqref{eq:backprob} and \eqref{eq:forprob} are unique under some hypotheses.
\begin{prop}
	The function $p_\nu(t,x;y)$ given in \eqref{eq:fundsol} is the unique global solution of \eqref{eq:fundback} (for fixed $x$) and \eqref{eq:fundfor} (for fixed $y$).
\end{prop}
\begin{proof}
	Let us first notice that, by Theorem \ref{thm:pnu} we know that $0\le p_\nu(t,x;y)\le 1$ and
	\begin{equation*}
	\sum_{x=0}^{+\infty}p_\nu(t,x;y)=\sum_{y=0}^{+\infty}p_\nu(t,x;y)=1,
	\end{equation*}
	hence $\Norm{p_\nu(t,\cdot;y)}{\ell^1}=\Norm{p_\nu(t,x;\cdot)}{\ell^1}=1$. Thus $p_\nu(t,x;y)$ is bounded in $\ell^1$ uniformly with respect to $t\ge 0$ for fixed $x$ or $y$. Since $\ell^1$ is continuously embedded in $\ell^2(\mm)$ (see Remark \ref{rmk:emb}), then $p_\nu(t,x;y)$ is uniformly bounded also in $\ell^2(\mm)$. Moreover, we have shown in Lemma \ref{lem:contlemma} that the operators $\cG$ and $\cL$ are continuous. Hence by \cite[Corollary $2$]{AscioneLeonenkoPirozzi2018} we can conclude that $p_\nu(t,x;y)$ is the unique global solution of \eqref{eq:fundback} and \eqref{eq:fundfor}.
\end{proof}
Now let us show the uniqueness of the solutions of the backward equation \eqref{eq:backprob}.
\begin{prop}
	Let $g \in \ell^\infty$ such that $g(x)=\sum_{n=0}^{+\infty} g_nQ_n(x)$. Then the strong solution $u(t,x)$ of \eqref{eq:backprob} is in $\ell^\infty$, hence also in $\ell^2(\mm)$, for any $t \ge 0$ and it is the unique global solution in $\ell^2(\mm)$.
\end{prop}
\begin{proof}
	First of all, let us observe that if $g \in \ell^\infty$, then $g \in \ell^2(\mm)$ too, so we are under the hypotheses of Theorem \ref{thm:strongback}. Moreover we have, by using Theorem \ref{thm:pnu} and Jensen inequality:
	\begin{align*}
	\begin{split}
	u^2(t,x)&=\left(\sum_{y=0}^{+\infty}p_\nu(t,y;x)g(y)\right)^2 \\&\le\sum_{y=0}^{+\infty}p_\nu(t,y;x)g^2(y)\le \Norm{g}{\ell^\infty}^2,
	\end{split}
	\end{align*}
	and then
	\begin{equation*}
	\sum_{x=0}^{+\infty}m(x)u^2(t,x)\le \Norm{g}{\ell^\infty}^2,
	\end{equation*}
	obtaining the uniform bound for $x \mapsto u(t,x)$. Hence $u(t,\cdot) \in \ell^2(\mm)$ for any $t \ge 0$. Moreover, since $\cG$ is a continuous operator, by \cite[Corollary $2$]{AscioneLeonenkoPirozzi2018}, it is the unique global solution of \eqref{eq:backprob}.
\end{proof}
We can also obtain the uniqueness of solutions of \eqref{eq:forprob}.
\begin{prop}\label{prop:uniq}
	Let $f:\N_0 \to \R$ be a function such that $\frac{f}{m} \in \ell^2(\mm)$ and $\frac{f(x)}{m(x)}=\sum_{n=0}^{+\infty}f_nQ_n(x)$. Then the strong solution $u(t,x)$ of \eqref{eq:forprob} is in $\ell^\infty$, hence in $\ell^2(\mm)$, and it is the unique global solution in $\ell^2(\mm)$.
\end{prop}
\begin{proof}
	Let us observe that
	\begin{equation*}
	u^2(t,x)=\left(\sum_{y=0}^{+\infty}p_\nu(t,x;y)f(y)\right)^2=\left(\sum_{y=0}^{+\infty}m(y)p_\nu(t,x;y)\frac{f(y)}{m(y)}\right)^2.
	\end{equation*}
	Thus, by Jensen inequality, we have
	\begin{equation*}
	u^2(t,x)\le \sum_{y=0}^{+\infty}m(y)p_\nu^2(t,x;y)\frac{f^2(y)}{m^2(y)}.
	\end{equation*}
	Now, from Theorem \ref{thm:pnu}, we know that $p_\nu(t,x;y)\le 1$, then
	\begin{equation*}
	u^2(t,x)\le \sum_{y=0}^{+\infty}m(y)\frac{f^2(y)}{m^2(y)}=\Norm{f/m}{\ell^2(\mm)}.
	\end{equation*}
	Finally, we have
	\begin{equation*}
	\sum_{x=0}^{+\infty}m(x)u^2(t,x)\le \Norm{f/m}{\ell^2(\mm)}\sum_{x=0}^{+\infty}m(x)=\Norm{f/m}{\ell^2(\mm)}
	\end{equation*}
	thus, since $\cL$ is a continuous operator, from \cite[Corollary $2$]{AscioneLeonenkoPirozzi2018} we have that $u(t,x)$ is the unique global solution of \eqref{eq:forprob}.
\end{proof}
\begin{rmk}
	The condition $f/m \in \ell^2(\mm)$ is stronger than $f \in \ell^2$ for any probability measure $\mm$ on $\mathbb{N}_0$. Indeed we can show that the following two properties
	\begin{itemize}
		\item[$a)$] $f \in \ell^2$;
		\item[$b)$] $f/\sqrt{m} \in \ell^2(\mm)$;
	\end{itemize}
	are equivalent: this can be done simply observing that
	\begin{equation*}
	\sum_{x=0}^{+\infty}f^2(x)=\sum_{x=0}^{+\infty}m(x)\left(\frac{f(x)}{\sqrt{m(x)}}\right)^2.
	\end{equation*}
	Moreover, if we consider the property
	\begin{itemize}
		\item[$c)$] $f/m \in \ell^2(\mm)$;
	\end{itemize}
	we can see that $c) \Rightarrow a)$. Indeed we have, since $m(x)\le 1$
	\begin{equation*}
	\sum_{x=0}^{+\infty}f^2(x)=\sum_{x=0}^{+\infty}m^2(x)\left(\frac{f(x)}{m(x)}\right)^2\le\sum_{x=0}^{+\infty}m(x)\left(\frac{f(x)}{m(x)}\right)^2.
	\end{equation*}
	However, if we consider $f(x)=\sqrt{m(x)}$, it is easy to verify that $f \in \ell^2$ but $f/m \not \in \ell^2(\mm)$.
\end{rmk}
\section{Limit distribution of $N_\nu(t)$}\label{sec7}
In this section we want to give some results on the limit distribution of $N_\nu(t)$. In particular we have
\begin{thm}\label{thm:statdist}
	Let $p_\nu(t,x;y)$ be the transition probability mass of $N_\nu(t)$. Then, given any initial probability mass $f$ such that $f/m \in \ell^2(\mm)$ and $f(x)/m(x)=\sum_{n \in \mathbb{N}_0}f_nQ_n(x)$, the probability mass of $N_\nu(t)$ asymptotically converges towards a Poisson measure, that is to say if $p_\nu(t,x)=\sum_{y \in \mathbb{N}_0}p_\nu(t,x;y)f(y)$, then
	\begin{equation*}
	\lim_{t \to +\infty}p_\nu(t,x)=m(x).
	\end{equation*}
\end{thm}
\begin{proof}
	By Corollary \ref{cor:stocrap} we know that $p_\nu(t,x)$ is solution of \eqref{eq:forprob}. Moreover, since $f/m \in \ell^2(\mm)$, we know that this solution is unique from Proposition \ref{prop:uniq}. Finally, from Theorem \ref{thm:strongfor}, we have that
	\begin{equation*}
	p_\nu(t,x)=m(x)\sum_{n \in \mathbb{N}_0}E_\nu(-bnt^\nu)f_nQ_n(x).
	\end{equation*}
	In particular we have
	\begin{equation*}
	p_\nu(t,x)=m(x)f_0Q_0(x)+m(x)\sum_{n \in \mathbb{N}}E_\nu(-bnt^\nu)f_nQ_n(x).
	\end{equation*}
	But $Q_0(x)=1$ and $f_0=\sum_{x \in \mathbb{N}_0}f(x)=1$ since $f$ is a probability mass. Thus we have
	\begin{equation*}
	p_\nu(t,x)=m(x)+m(x)\sum_{n \in \mathbb{N}}E_\nu(-bnt^\nu)f_nQ_n(x).
	\end{equation*}
	Now, by Cauchy-Schwartz inequality, the fact that $E_\nu(-bnt^\nu)\le 1$ and the duality formula for $C_n(x,\alpha)$, we obtain
	\begin{align*}
	\begin{split}
	\sum_{n \in \mathbb{N}}|E_\nu(-bnt^\nu)f_nQ_n(x)|&\le\sum_{n \in \mathbb{N}}|f_nQ_n(x)|\\&\le \Norm{f/m}{\ell^2(\mm)}\left(\sum_{n \in \mathbb{N}}\frac{\alpha^n}{n!}C_n^2(x,\alpha)\right)^{\frac{1}{2}}\\&\le e^{\frac{\alpha}{2}}\Norm{f/m}{\ell^2(\mm)}d_x^2
	\end{split}
	\end{align*}
	hence the second series totally converges. Thus we can take the limit inside the series and, since $\lim_{t \to +\infty}E_\nu(-bnt^\nu)=0$, we have
	\begin{equation*}
	\lim_{t \to +\infty}p_\nu(t,x)=m(x)+m(x)\sum_{n \in \mathbb{N}}\lim_{t \to +\infty}E_\nu(-bnt^\nu)f_nQ_n(x)=m(x).
	\end{equation*}
\end{proof}
From Theorem \ref{thm:statdist} we know that whatever is the distribution of $N_\nu(0)$, the limit distribution of $N_\nu(t)$ is always $\mm$. Moreover, we can show that $\mm$ is an invariant one-dimensional distribution for $N_\nu(t)$, that is to say that if $N_\nu(0)$ has distribution $\mm$, then $N_\nu(t)$ admits $\mm$ as distribution for any $t>0$.
\begin{prop}
	Suppose $N_\nu(0)$ has distribution $\mm$. Then for any $t>0$, $N_\nu(t)$ has distribution $\mm$.
\end{prop}
\begin{proof}
	Let us observe that the density of $N_\nu(t)$ is given by
	\begin{equation*}
	p_\nu(t,x)=\sum_{y \ge 0}p_\nu(t,x;y)m(y).
	\end{equation*}
	From Theorem \ref{thm:pnu} we have that
	\begin{equation*}
	p_\nu(t,x;y)=m(x)\sum_{n=0}^{+\infty}E_\nu(-bnt^\nu)Q_n(x)Q_n(y)
	\end{equation*}
	then we have
	\begin{align*}
	p_\nu(t,x)&=\sum_{y \ge 0}\left(m(x)\sum_{n=0}^{+\infty}E_\nu(-bnt^\nu)Q_n(x)Q_n(y)\right)m(y)\\&=m(x)\sum_{n=0}^{+\infty}E_\nu(-bnt^\nu)Q_n(x)\left(\sum_{y=0}^{+\infty}Q_n(y)m(y)\right).
	\end{align*}
	Recalling that $Q_0(y)\equiv 1$ we have that
	\begin{equation*}
	\sum_{y=0}^{+\infty}Q_n(y)m(y)=\sum_{y=0}^{+\infty}Q_0(y)Q_n(y)m(y)=\delta_{n,0}
	\end{equation*}
	hence
	\begin{equation*}
	p_\nu(t,x)=m(x)\sum_{n=0}^{+\infty}E_\nu(-bnt^\nu)Q_n(x)\delta_{n,0}=m(x).
	\end{equation*}
\end{proof}
However, since $N_\nu(t)$ is not Markovian, this Proposition does not guarantee the stationarity of the process when $N_\nu(0)$ admits $\mm$ as distribution. However, it is still possible to compute the autocovariance function of the process $N_\nu(t)$.
\begin{prop}
	Suppose $N_\nu(t)$ admits $\mm$ as initial distribution. Then
	\begin{equation}\label{eq:cov}
	\Cov(N_\nu(t),N_\nu(s))=\alpha\left(E_\nu(-bt^\nu)+\frac{b\nu t^\nu}{\Gamma(1+\nu)}\int_{0}^{\frac{s}{t}}\frac{E_\nu(-bt^\nu(1-z)^{\nu})}{z^{1-\nu}}dz\right).
	\end{equation}
\end{prop}
We omit the proof of this Proposition since it is identical to the one in \cite{Leonenko2013b}, after observing that if $N_1(t)$ admits $\mm$ as initial distribution, then $N_1(t)$ is stationary and, from \eqref{eq:condmean},
\begin{equation*}
\Cov(N_1(t),N_1(0))=\alpha e^{-bt}.
\end{equation*}
Remark $3.2$ and $3.3$ of \cite{Leonenko2013b} easily apply also to our process $N_\nu(t)$. Indeed, since in this case $N_\nu(t)$ is distributed as $N_\nu(0)$, then the variance $\D[N_\nu(t)]=\D[N_\nu(0)]=\alpha$, which can be obtained from \eqref{eq:cov} when $t=s$ with the same calculations as in \cite[Remark $3.2$]{Leonenko2013b}. Moreover, $N_1(t)$ exhibits long-range dependence, while, with the same calculations of \cite[Remark $3.3$]{Leonenko2013b}, one can show that $\Cov(N_\nu(t),N_\nu(s))$ decays as a power of $t$, hence it exhibits short-range dependence.
\appendix

\end{document}